\documentclass[preprint]{elsarticle}

\usepackage[english]{babel}
\usepackage[T1]{fontenc}
\usepackage[utf8]{inputenc}
\usepackage{lmodern}

\usepackage{amsmath,amssymb,amsfonts,amsthm}
\usepackage{mathtools,commath}
\usepackage{siunitx}
\usepackage{overpic}
\usepackage{graphicx}
\usepackage[outline]{contour}
\contourlength{1pt}
\usepackage[draft]{fixme}
\fxsetup{envlayout=color,targetlayout=color}
\fxusetheme{colorsig}

\newcommand{\SE}[1][3]{\operatorname{SE}(#1)}
\newcommand{\D}{\mathbb{D}}
\newcommand{\bH}{\mathbb{H}}
\newcommand{\qi}{\mathbf{i}}
\newcommand{\qj}{\mathbf{j}}
\newcommand{\qk}{\mathbf{k}}
\newcommand{\eps}{\varepsilon}
\newcommand{\cj}[1]{#1^\star}
\newcommand{\crossM}[1]{[#1]_\times}
\newcommand{\ej}[1]{#1_\eps}
\newcommand{\cay}[1]{C(#1)}

\newcommand{\vv}[1]{#1}
\newcommand{\vvi}{\vv{i}}
\newcommand{\vvj}{\vv{j}}
\newcommand{\vvk}{\vv{k}}
\newcommand{\vvl}{\vv{l}}

\newcommand{\Ng}{f^-}
\newcommand{\Nh}{f^+}
\newcommand{\NTg}{\tilde{f}^-}
\newcommand{\NTh}{\tilde{f}^+}

\newcommand{\quadric}[1]{\mathcal{#1}}
\newcommand{\QQ}{\quadric{Q}}
\newcommand{\SQ}{\quadric{S}}

\newtheorem{proposition}{Proposition}
\newtheorem{theorem}{Theorem}
\theoremstyle{definition}
\newtheorem{definition}{Definition}
\theoremstyle{remark}
\newtheorem{example}{Example}
\newtheorem{remark}{Remark}
\newtheorem*{remark*}{Remark}

\usepackage{hyperref}
\hypersetup{
  pdfauthor={Daniel Huczala, Martin Pfurner, Hans-Peter Schröcker, Gudrun Szewieczek},%
  pdftitle={Bistable Quad-Nets Composed of Four-Bar Linkages},%
  pdfborder={0 0 0},%
  colorlinks=true,%
}

\begin{document}

\begin{frontmatter}
\journal{}

\title{Bistable Quad-Nets Composed of Four-Bar Linkages}

\author[1]{Gudrun Szewieczek}
\ead{gudrun.szewieczek@uibk.ac.at}

\author[2]{Daniel Huczala}
\ead{daniel.huczala@snu.ac.kr}

\author[1]{Martin Pfurner}
\ead{martin.pfurner@uibk.ac.at}

\author[1]{Hans-Peter Schröcker}
\ead{hans-peter.schroecker@uibk.ac.at}

\address[1]{University of Innsbruck, Department of Basic Sciences in Engineering Sciences, Innsbruck, Austria}
\address[2]{Seoul National University, Robotics Laboratory, Seoul, Korea}

\begin{abstract}
  We study a novel type of mechanical structures, composed of spatial four-bar linkages, that are bistable, that is, they allow for two distinct configurations. These structures have an interpretation as quad nets in the Study quadric which we use to prove existence of assemblies with an unbounded number of links and joints. We propose a purely geometric construction of such objects, starting from infinitesimally flexible quad nets in Euclidean space and applying Whiteley de-averaging. This point of view situates the problem within the broader framework of discrete differential geometry and enables the construction of bistable structures from well-known classes of quad nets, such as discrete minimal surfaces. In contrast to many other construction methods for bistable structures, our approach does not rely on numerical optimization and it allows for simple control of relevant geometric parameters such as axis positions and snap angles.
\end{abstract}

\begin{keyword}
  four-bar linkage \sep
  bistable structure \sep
  quadric net \sep
  infinitesimal flexibility \sep
  discrete surface \sep
  discrete Koenigs net
  \MSC[2020]{%
    52C25 
    70B15 
    53A17 
  }
\end{keyword}

\end{frontmatter}

\section{Introduction}
\label{sec:introduction}

\subsection{Background} Bistable or multistable structures, sometimes also called ``snapping'', are an important topic of theoretical and applied research for at least three decades, often with a distinct interdisciplinary flavor that connects mechanism science, robotics, material science, bionics, physics, mathematics, computer science, etc. We do not even attempt to provide a full literature overview here but instead refer to the recent survey paper \cite{zhang24}. There, bistable structures are defined by the existence of ``two distinct stable states'' that require a ``minimal energy input'' to effect a ``snap-through from the `natural' (initial) state to the `everted' state''.

\subsection{Motivation and Related Work}

In this text we propose a novel bistable structure, \emph{snapping four-bar nets.} These are composed of four-bars that are connected by rigid links according to the combinatorics of a quad net with revolute joints sitting on edges. In the terminology of mechanism science, snapping four-bar nets are multi-looped kinematic chains with (rigid) revolute joints. They have zero degrees of freedom but their configuration space consists of \emph{precisely two points.} It is possible to assemble them in either of two distinct configurations but any transition between these two states is impossible in theory. Practically though, if the two configurations are ``close'' enough, due to material flexibility and joints imperfection, a snap-through can be induced by an energy input. Importantly, the transition requires deformation of the bodies, i.e. it is not a rigid-body motion. Classic techniques of kinematics can be applied to determine the two stable configurations but are insufficient to model the snap-through.

Snapping mechanisms offer several practical benefits that make them attractive across engineering domains. The core advantage lies in mechanical simplicity: they transition between states quickly, often requiring only small trigger forces, and once snapped, they hold their position without any continuous energy input. This passive stability is particularly valuable in deployable structures, even in architecture such as foldable bridges \cite{Maleczek_2019} or reconfigurable robotics (reviewed in \cite{cao21}). The integration with simple binary actuators like pneumatic valves can make applications straightforward and cost-effective.

In contrast to many state-of-the-art examples, snapping four-bar nets are truly spatial, exhibit no obvious symmetries and allow for clear geometric constructions that do not require optimization. While four-bar mechanisms (at least flexible ones) and structures composed from them are classical objects of study in mechanism science \cite{chen08,gallet16,li18} or \cite[Chapter~5]{you11}, these novel structures do not fit exactly into the classification given in \cite[Section~2]{zhang24} which names structures based on bars and beams, curved shells, folding techniques based on origami, and composite materials. Our snapping four-bar nets are reminiscent of some origami structures and also related to ``bistable scissor structures'' \cite{arnouts20,arnouts20b} which are subsumed in the category of non-traditional designs in \cite{zhang24}. Both, origami and scissor structures can be modeled as linkages in the sense of mechanism science: They are composed of rigid bodies and revolute axes that can be grouped into \emph{spherical} or \emph{planar} linkages. In contrast, the snapping four-bar nets we propose allow for truly spatial axes arrangements that, when reduced to stripes in the quad net, produce the typical combinatorics of scissor linkages. In the flexible case, these have been used for the construction of Kempe linkages to planar and spatial rational curves \cite{gallet16,li18}.

With this in mind, let us summarize and emphasize two important aspects that distinguish snapping four-bar nets from most bistable structures known in literature:
\begin{itemize}
\item Snapping four-bar nets allow for \emph{exactly two configurations} in a \emph{precise mathematical sense.} Their calculation requires only simple geometric constructions or algebraic computations. No energy optimization is involved but the construction leaves plenty degrees of freedom to optimize for secondary goals.
\item As a consequence of this exact geometric/algebraic construction, no energy is stored in either of the two configurations of the snapping four-bar net. A snap will only occur if energy is induced into the system. This is not typically the case for the structures discussed in the survey article \cite{cao21}. We view it is a neutral property that may be advantageous or not, depending on the particular application.
\end{itemize}

\subsection{Main Results and Outline}

The basic building blocks for snapping four-bar nets are \emph{snapping four-bar mechanisms.} These have been introduced in \cite{wunderlich71} and received some recent attention, either as simple examples and test cases for numeric ``snappability'' questions \cite{nawratil20,nawratil22,zhou24} or -- in a rather specific situation -- because of their relevance to Henrici's flexible hyperboloid model \cite{stachel25}. We review some of their properties in Section~\ref{sec:four-bars}, address known construction methods, and explain the difficulties to apply them to larger structures.

Our main theoretical contribution is existence of snapping four-bar nets, based on a kinematic interpretation of snapping four-bars: The two configurations can be thought of as fixed and as moving axode of a ``discrete Ribaucour motion'', a natural discretization of smooth Ribaucour motions that are characterized by having zero instantaneous pitch \cite{selig16}. In this way, snapping four-bars are related to rotation quadrilaterals \cite{schroecker10,schroecker12} and snapping four-bar nets appear as quad nets in the Study quadric \cite[Section~11]{selig05}. The existence of these quadric nets is guaranteed for net dimension $d \le 6$ and also ensures existence of corresponding snapping four-bar nets. This we describe in Section~\ref{sec:snapping-nets}.

While this resolves theoretical issues related to the existence of snapping four-bar nets, practical examples are still difficult to realize as there is but little control over the position of individual revolute axes and over snap angles. We address this in Section~\ref{sec:surface-design} for the case $d = 2$, again by methods from discrete differential geometry. Starting from an infinitesimally flexible quad net, a well-studied object in discrete differential geometry, we use Whiteley de-averaging \cite{Whiteley2004} to obtain a pair of discrete quad surfaces with congruent corresponding faces. It gives rise to a discrete Ribaucour motion of dimension two whose discrete moving and fixed axodes produce the two configurations of a snapping four-bar net with a \emph{surface-like appearance.} This approach is quite simple and straightforward but may seem rather particular. We prove that this is not the case and, in fact, a large class of snapping four-bar nets can be obtained in that way.

\section{Snapping Four-Bars}
\label{sec:four-bars}

A \emph{four-bar linkage} (or \emph{four-bar,} for short) is a mechanical structure composed of four revolute joints $R_0$, $R_1$, $R_2$, $R_3$ which are arranged in cyclic order and such that any two consecutive axes are connected by a rigid link. By the mobility criterion of Grübler-Kutzbach-Chebychev \cite[Chapter~13]{hunt78}, a spatial four-bar with general design is expected to have a mobility of $-2$. This means that, generically, it cannot be assembled if the relative positions of consecutive revolute axes are prescribed. However, there exist special axes configurations that allow for one, two, or even infinitely many configurations. The interpretation of 2R chains in the Study quadric model of space kinematics as intersection of the Study quadric $\SQ$ with a sub-space of dimension three in projective space $\mathbb{P}^7$ in \cite{brunnthaler05} can be used to show that these are, indeed, all possible cases. The closed 4R loop is obtained as intersection of two projective three spaces with the Study quadric~$\SQ$.

The case of only one possible configuration is either \emph{rigid} or \emph{infinitesimally flexible (shaky),} the latter if and only if the axes are taken from one family of rulings on a hyperboloid, c.f. \cite{wunderlich71} or \cite[Section~4.2]{stachel25}. The case of infinitely many configurations yields a \emph{mobile} linkage. It is either planar (all axes are parallel), spherical (all axes are concurrent), or a Bennett linkage \cite[Section~11.4.1]{mccarthy11}. The case we are interested in is that of two configurations. The corresponding four-bars are called \emph{snapping} or \emph{bistable.} Note that we count mobile and shaky four-bars as special cases among the snapping ones, similar to~\cite{stachel25}.

\subsection{Constructions of Snapping Four-Bars}
\label{sec:construction}

Literature knows several constructions of snapping four-bars and we briefly review them here. It turns out that they are unsuitable for the generation of larger structure because of little control on the position of axes. The original construction of Wunderlich \cite{wunderlich71} designates one link as fixed. It contains two fixed revolute axes, say $R_0$ and $R_3$. Wunderlich then prescribes two different configurations $(R'_1,R'_2)$, $(R''_1,R''_2)$ of the remaining (moving) axes $R_1$ and $R_2$ and constructs $R_0$ and $R_3$ such that $R'_1$ and $R''_1$ correspond in a rotation (or a translation in special cases) around $R_0$ and $R'_2$ and $R''_2$ correspond in a rotation around~$R_3$.

Wunderlich's construction uses the obvious observation that two incongruent configurations of a four-bar can be placed such that the axes $R_3$ and $R_0$ coincide. This is not only true for $R_3$ and $R_0$ but for any pair of consecutive axes and we will make use of this in the forthcoming Section~\ref{sec:rotation-quadrilaterals} which relates snapping four-bars to rotation quadrilaterals in the terminology of \cite{schroecker12}. This connection gives rise to several construction methods of snapping four-bars. Some of them are described in \cite{schroecker10,schroecker12}, for example the non-unique decomposition of a helical displacement into two rotations which can be inferred easily from \cite[Figure~10.7]{mccarthy11}.

A common feature of all of these constructions is the difficulty to control axis positions. This is particularly disturbing in the design of larger snapping structures. \cite[Section~4]{stachel25} alleviates these problems to a certain extend as it allows to prescribe two isometric quadrilaterals whose vertices contain the axes of two incongruent realizations. Still, it is difficult to ensure that closure conditions in snapping quad nets are fulfilled.

In the next section we build on the kinematic interpretation of snapping four-bars as rotation quadrilaterals and we extend it. We present yet another construction of rotation quadrilaterals in terms of dual quaternions and the Study quadric model of space kinematics with the main purpose to show \emph{existence} of arbitrarily large snapping structures composed of four-bars. Later, in Section~\ref{sec:surface-design}, we will use this interpretation to construct snapping four-bar nets directly from discrete surface-like structures.

\subsection{Snapping Four-Bars and Rotation Quadrilaterals}
\label{sec:rotation-quadrilaterals}

Consider a snapping four-bar with two incongruent configurations, given by axes quadruples $(R_0,R_1,R_2,R_3)$ and $(R'_0,R'_1,R'_2,R'_3)$, respectively. We additionally set $R_4 \coloneqq R_0$ and $R'_4 \coloneqq R'_0$. The angle and distance between any two consecutive axes pairs $(R_i,R_{i+1})$ and $(R'_i,R'_{i+1})$ are equal for any $i \in \{0,1,2,3\}$. This means that
\begin{itemize}
\item the second configuration can be rigidly transformed to a pose were $R_i$ and $R_{i+1}$ coincide with $R'_i$ and $R'_{i+1}$, respectively, and
\item any consecutive two of these poses correspond in a rotation $\varrho_i$ around their common axis $R_i$.
\end{itemize}
This is illustrated in Figure~\ref{fig:rotations}.

\begin{figure}
  \centering
  \includegraphics{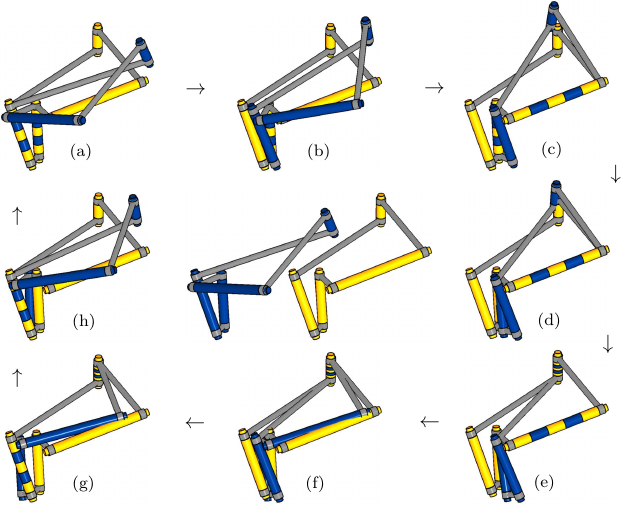}
  \caption{Two configurations (orange and blue) of a snapping four-bar (\emph{center}), in poses where two axes are aligned (a, c, e, g) and in intermediate poses (b, d, f, h). Coinciding axes are indicated by using both colors.}
  \label{fig:rotations}
\end{figure}

We thus see that a snapping four-bar gives rise to four relative rotations $\varrho_0$, $\varrho_1$, $\varrho_2$, $\varrho_3$ whose composition is the identity (c.f.\cite[Theorem 7]{stachel25}):
\begin{equation}
  \label{eq:rotations-identity}
  \varrho_3 \circ \varrho_2 \circ \varrho_1 \circ \varrho_0 = \text{id.}
\end{equation}
In \cite{schroecker10,schroecker12}, this sequence of rotations was called a \emph{rotation quadrilateral.} The two configurations of the snapping four-bar can be recovered from any rotation quadrilateral as ``discrete fixed and moving axode'', respectively. The important insight to be gained at this point is that \emph{snapping four-bars and rotation quadrilaterals mutually determine each other.} This statement has to be taken with a grain of salt though as snapping four-bars occur generically but infinitesimally or continuously flexible four-bars are possible as well. Also note that limiting cases may include translations instead of rotations. These give rise to bistable structures composed of revolute and prismatic joints.

\subsection{Rotation Quadrilaterals and Dual Quaternions}

Rotation quadrilaterals have an elegant interpretation in the dual quaternion model of space kinematics. We briefly describe it here and provide a numeric example. The dual quaternion interpretation will be used in Section~\ref{sec:quadric-nets} to show existence of arbitrarily large snapping structures composed of four-bars.

At first, a very brief introduction to dual quaternions and space kinematics is necessary. A more in-depth treatment can be found for example in \cite[Section~9.3 and Chapter~11]{selig16}. A dual quaternion~$p$ can be seen as a quaternion over the ring of dual numbers or, equivalently, as $p = a + \eps c$ where both $a$ and $c$ are ordinary quaternions and the new symbol $\eps \neq 0$ commutes with everything and satisfies $\eps^2 = 0$. We denote the set of quaternions by $\bH$, the set of dual quaternions by $\D\bH$, and we extend quaternion conjugation to $\D\bH$ as $\cj{p} \coloneqq \cj{a} + \eps \cj{c}$ where $\cj{a}$ and $\cj{c}$ are the usual conjugates of $a$ and $c$ as quaternions. We also define the $\eps$-conjugate of $p$ as $p_\eps \coloneqq a - \eps c$.

With this, $p\cj{p} = a\cj{a} + \eps(a\cj{c}+c\cj{a})$ is a dual number. Dual quaternions $p = a + \eps c$ satisfying $p\cj{p} \in \mathbb{R}$ or, equivalently,
\begin{equation}
  \label{eq:study-condition}
  a\cj{c} + c\cj{a} = 0
\end{equation}
are sometimes called \emph{Study quaternions} and \eqref{eq:study-condition} is called the \emph{Study condition.} Study quaternions constitute a well-known model for $\SE$. Identifying the point $(x_1,x_2,x_3) \in \mathbb{R}^3$ with the dual quaternion $x = 1 + \eps(x_1\qi + x_2\qj +x_3\qk)$, the action
\[
  \mathbb{R}^3 \to \mathbb{R}^3,\quad
  x \mapsto \frac{p_\eps x \cj{p}}{p\cj{p}}
\]
is well-defined unless $a = 0$ and describes a rigid body displacement. Clearly, non-zero scalar multiples of $p$ yield the same displacement. This provides a map from $\SE$ into the projective space $\mathbb{P}(\D\bH) \cong \mathbb{P}^7(\mathbb{R})$ over the dual quaternions. The image of $\SE$ is precisely the \emph{Study quadric $\SQ$} given by \eqref{eq:study-condition} minus the projective space given by $a = 0$. The two most relevant properties of Study's kinematic map for us are:
\begin{enumerate}
\item The composition of rigid body displacements corresponds to dual quaternion multiplication.
\item The relative displacement between two invertible Study quaternions $p$, $q \in \D\bH$ is a rotation or a translation if and only if the line joining them (when viewed as points in the Study quadric) is contained in the Study quadric~$\SQ$. The case of a straight line in $\SQ$ corresponding to a translation is non generic and, for reasons of simplicity, we ignore it in this text.
\end{enumerate}

\begin{example}
  We illustrate the construction of snapping four-bars from rotation quadrilaterals in an example. Consider the four dual quaternions
  \[
    p_0 = 1,\quad
    p_1 = 1 + \qk - \eps(\qi + 2\qj),\quad
    p_2 = 1 - \qj + 2\qk - \eps(3+\qj-\qk),\quad
    p_3 = 1 - \qj + \eps(\qi + \qk).
  \]
  They satisfy
  \[
    \{p_0\cj{p}_0, p_1\cj{p}_1, p_2\cj{p}_2, p_3\cj{p}_3\} \subset \mathbb{R}
  \]
  as well as
  \[
    \{p_0\cj{p_1} + p_1\cj{p_0},\
      p_1\cj{p_2} + p_2\cj{p_1},\
      p_2\cj{p_3} + p_3\cj{p_2},\
      p_3\cj{p_0} + p_0\cj{p_3}\} \subset \mathbb{R}.
  \]
  The former is the Study condition \eqref{eq:study-condition}, the later is the condition that the relative displacements
  \begin{equation}
    \label{eq:relative-rotations}
    r_0 \coloneqq p_1\cj{p}_0,\quad
    r_1 \coloneqq p_2\cj{p}_1,\quad
    r_2 \coloneqq p_3\cj{p}_2,\quad
    r_3 \coloneqq p_0\cj{p}_3
  \end{equation}
  are rotations. This is also confirmed by the explicit expressions
  \begin{alignat*}{2}
    r_0 &= 1 + \qk - \eps(\qi+2\qj),&\quad
    r_1 &= 3 + \qi - \qj + \qk - \eps(2\qi-3\qj-5\qk),\\
    r_2 &= 2(1+\qi-\qk)+\eps(\qi+6\qj+\qk),&\quad
    r_3 &= 1 + \qj - \eps(\qi+\qk),
  \end{alignat*}
  combined with the remark that the Study quaternion $p$ describes a rotation if and only if its dual part $c$ satisfies $c + c^* = 0$ \cite[Section~11.2.1]{selig05}. By construction
  \[
    r_3 r_2 r_1 r_0 = 24
  \]
  is real so that \eqref{eq:rotations-identity} is indeed fulfilled. The Plücker coordinates of the revolute axes in the fixed space are the respective vector parts $\frac{1}{2}(r_i-r^*_i)$ of $r_0$, $r_1$, $r_2$, and $r_3$ \cite{radavelli14}:
  \begin{alignat*}{2}
    R_0 &= \qk - \eps (\qi+2\qj),&\quad
    R_1 &= \qi - \qj + \qk - \eps(2\qi-3\qj-5\qk),\\
    R_2 &= 2(\qi-\qk) + \eps(\qi+6\qj+\qk),&\quad
    R_3 &= \qj - \eps(\qi+\qk).
  \end{alignat*}
  They determine the Denavit-Hartenberg parameters $a_i$ (distances), $d_i$ (offsets), and $\alpha_i$ (angles) as follows:
  \begin{equation}
    \label{eq:dh-parameters}
    \begin{gathered}
      a_0 = 3\sqrt{2},\quad
      a_1 = \tfrac{3}{2}\sqrt{6},\quad
      a_2 = \tfrac{3}{2}\sqrt{2},\quad
      a_3 = 3;\\
      d_0 = 1,\quad
      d_1 = -\tfrac{\sqrt{3}}{2},\quad
      d_2 = \tfrac{\sqrt{2}}{2},\quad
      d_3 = -\tfrac{1}{2};\\
      \alpha_0 = \arccos(\tfrac{\sqrt{3}}{3}),\quad
      \alpha_1 = \tfrac{\pi}{2},\quad
      \alpha_2 = \tfrac{\pi}{2},\quad
      \alpha_3 = \tfrac{\pi}{2}.
    \end{gathered}
  \end{equation}
  The second configuration of revolute axes is obtained by acting with $\cj{r}_1$ on $R_2$ and by $r_0$ on $R_3$. More precisely, it is given by the lines $Q_0 \coloneqq R_0$, $Q_1 \coloneqq R_1$, as well as
  \begin{equation*}
    \begin{aligned}
      Q_2 &= \ej{(\cj{r}_1 R_2 r_1)} = -24(\qj+\qk) - \eps(48\qi +
            36\qj-36\qk)\\
      \text{and}\quad
      Q_3 &= \ej{(r_0 R_3 \cj{r}_0)} = 2\qi - \eps(2\qj - 4\qk).
    \end{aligned}
  \end{equation*}
  Clearly, $Q_0$, $Q_1$, $Q_2$, and $Q_3$ determine the same Denavit-Hartenberg parameters~\eqref{eq:dh-parameters}.
\end{example}

To summarize: In the framework of dual quaternions, a snapping four-bar can be equivalently described by four dual quaternions $p_0$, $p_1$, $p_2$, $p_3 \in \SQ \subset \mathbb{P}^7(\mathbb{R})$ lying on the Study quadric such that the relative displacements \eqref{eq:relative-rotations} are rotations. These four relative displacements~$r_0$, $r_1$, $r_2$, $r_3$ constitute the corresponding rotational quadrilateral.

\section{Snapping Four-Bar Nets}
\label{sec:snapping-nets}

In analogy with snapping four-bars, in this subsection we generalize the definition to higher-order combinatorial structures, organized according to the lattice~$\mathbb{Z}^d$.

\begin{definition}
  A \emph{four-bar net} is a mechanical structure composed of revolute joints assigned to the edges of the net~$\mathbb{Z}^d$. Two revolute joints are connected by a rigid link precisely if their corresponding edges share a vertex. A four-bar net is said to be \emph{snapping} if it admits a second, incongruent configuration.
\end{definition}

Thus, a snapping four-bar net is a structure composed of snapping four-bars, such that not only each individual four-bar but the entire structure is bistable. Examples of CAD models of such structures are presented later in Section~\ref{sec:3d-print} and shown in Figure~\ref{fig:prints}.

Note that, as in~\cite{stachel25} for snapping four-bars, the definition of snapping four-bar nets also allows individual four-bars, or even all four-bars, to be flexible. To avoid additional technicalities, we do not exclude these special configurations in our analysis.

In order to describe snapping four-bar nets formally, we work with discrete maps defined on the common domain~$\mathbb{Z}^d$ but with values in different image spaces. In our setting, two such image spaces are of particular importance: first, discrete maps into the Study quadric~$\SQ$, also referred to as \emph{discrete $d$-dimensional motions;} and second, discrete maps into~$\mathbb{R}^3$. For $d=2$, the latter are called \emph{discrete surfaces} or \emph{quad nets in $\mathbb{R}^3$} and are fundamental objects in discrete differential geometry, where they are interpreted as discrete surface parametrizations as well as surface transformations~\cite{bobenko08}.

Our first goal is to extend the interpretation of a single snapping four-bar as a quadrilateral in the Study quadric, introduced in Section~\ref{sec:four-bars}, to the combinatorics of snapping four-bar nets. This will also serve to introduce the notation used throughout the text for discrete maps.

\begin{definition}
  A \emph{quad net} of dimension $d$ in projective space $\mathbb{P}^n$ is a map $p\colon \mathbb{Z}^d \to \mathbb{P}^n$. The image of $\vvi \in \mathbb{Z}^d$ is also denoted by $p_\vvi \coloneqq p(\vvi)$. It is called a \emph{vertex} of the net. An \emph{edge} is the connecting line of two vertices $p_\vvi$ and $p_\vvj$ where $\vvi-\vvj$ is a standard basis vector of~$\mathbb{R}^d$.
\end{definition}

\begin{remark*}
  The elements $\vvi$, $\vvj \in \mathbb{Z}^d$ etc.\ are actually multi-indices while they look like indices in our notation. This helps to lighten notation and confusion is not to be expected. Occasionally, when considering nets of dimension $d = 2$, we will also use a double index notation and write $p_{m,n}$ instead of $p_\vvi$ where $\vvi = (m,n) \in \mathbb{Z}^2$.
\end{remark*}

\begin{definition}
  A discrete map~$p\colon \mathbb{Z}^d \mapsto \SQ \subset \mathbb{P}(\D\bH) \cong \mathbb{P}^7(\mathbb{R})$ is called an $\SQ$-net if, for every edge~$(i,j)$, the relative displacement
  \begin{equation*}
    r_{ij} \coloneqq p_j\cj{p}_i
  \end{equation*}
  is a rotation. In this case, the discrete map~$r$ of relative displacements is referred to as a \emph{rotation net} or a \emph{discrete rolling motion}.
\end{definition}

An $\SQ$-net fits naturally into a more general framework which is well-known from discrete differential geometry, for example \cite{Hoffmann2025}:

\begin{definition}
  Given a hyper-quadric $\QQ \subset \mathbb{P}^n$, we call a quad net $p$ in~$\mathbb{P}^n$ a \emph{$\QQ$-net} if its vertices and its edges are contained in~$\QQ$.
\end{definition}

An $\SQ$-net generically gives rise to a snapping four-bar net. The argument is the same as in Section~\ref{sec:four-bars}: The two configurations of the snapping four-bar net can be regarded as the fixed and the moving axode, respectively, of the discrete rolling motion given by the $\SQ$-net, where any two neighboring poses (vertices) are connected by a relative rotation (an edge).

This simple kinematic argument automatically implies bistability of the whole structure. Its strength should not be underestimated. Not only the snappability of each individual four-bar is required but also the compatibility of all snaps in closed cycles of the mechanism. In order to better appreciate this, we study the topology of the underlying mechanism.

\subsection{Mechanisms on Quad Graphs}

Neighboring vertices $p_\vvi$ and $p_\vvj$ of the $\SQ$-net~$p$ in the Study quadric $\SQ$ describe two poses that correspond in a rotation. We should think of them as links of the underlying mechanism. Any two links are connected by a revolute joint if and only if the two vertices $p_\vvi$, $p_\vvj \in \SQ$ are connected by an edge (which lies in the Study quadric by definition). In other words, the revolute joints correspond to the edges of the $\SQ$-net and it is natural to view the $\SQ$-net as the mechanism's ``linkgraph'' in the sense of~\cite{li18}.

The relation between linkgraph and mechanism is illustrated at hand of a simple example in~Figure~\ref{fig:scissor}. It displays a planar scissor linkage on top of its linkgraph which is composed of two neighboring quads. Vertices of the linkgraph are indicated by filled dots and they are labeled. Revolute joints (empty dots) sit on edges of the linkgraph and are connected by links. The number of revolute joints attached to a link equals the valence of the corresponding linkgraph vertex. The linkgraph vertex $p_{0,0}$ represents the fixed frame, the opposite vertex $p_{2,1}$ represents the moving frame.

\begin{figure}
  \centering
  \includegraphics[scale=1.0]{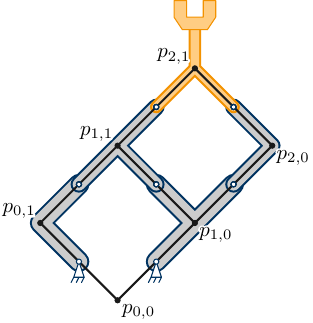}
  \caption{Schematic sketch of a planar scissor linkage. Its linkgraph consists of two quads. Note that graph vertices (black dots) correspond to links while edges correspond to revolute joints (white dots).}
  \label{fig:scissor}
\end{figure}

The planar mechanism in Figure~\ref{fig:scissor} only serves illustration purposes. It is movable with one degree of freedom. If the underlying quad net (consisting of six vertices and seven edges) in $\SQ$ is generic, it gives rise to a bistable scissor like structure composed of two \emph{spatial four-bars.}

\subsection{Quad Nets in Quadrics}
\label{sec:quadric-nets}

We now study existence and computation of $\QQ$-nets. Our results are applicable to quadrics in projective spaces of any dimension but we only formulate them for the Study quadric $\QQ = \SQ \subset \mathbb{P}^7(\mathbb{R})$ and argue that quadric nets defined on $\mathbb{Z}^d$ exist for $d \le 6$.

With
\[
  s(x,y) \coloneqq x_0y_4 + x_1y_5 + x_2y_6 + x_3y_7 + x_4y_0 + x_5y_1 + x_6y_2 + x_7y_3
\]
being the bilinear form to the Study quadric $\SQ$, the $\SQ$-net conditions are:
\begin{equation}
  \label{eq:Q-net-conditions}
  \begin{aligned}
    s(p_\vvi,p_\vvi) &= 0 \quad \text{for any $\vvi \in \mathbb{Z}^d$,}\\
    s(p_\vvi,p_\vvj) &= 0 \quad \text{for any edge $(\vvi,\vvj)$.}
  \end{aligned}
\end{equation}

Thus, an $\SQ$-net $p$ can be constructed iteratively as follows:
\begin{itemize}
\item Arbitrarily prescribe the value of $p$ at $(0,0,\ldots,0)\in\mathbb{Z}^d$. It is subject to only one quadratic condition, the Study condition.
\item Construct the values of $p$ along the $d$~coordinate axes by subsequently solving systems of one linear and one quadratic equation of type \eqref{eq:Q-net-conditions}.
\item For $k = 2, 3, \ldots, d$, construct the values of $p$ along the $k$-dimensional coordinate planes by subsequently solving systems of $k$ linear and one quadratic equation of type \eqref{eq:Q-net-conditions}. Generically, solutions will exist because the number of equations does not exceed the number of dimensions of the ambient space.
\end{itemize}

Note that non-generic straight lines in $\SQ$ that correspond to translations should be avoided in this construction.

\begin{figure}
  \centering
  \includegraphics[]{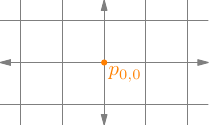}\hfill%
  \includegraphics[]{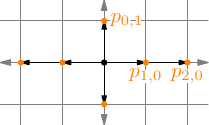}\hfill%
  \includegraphics[]{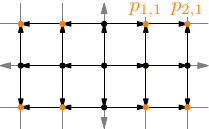}
  \caption{Construction of a $\QQ$-net on $\mathbb{Z}^2$. We prescribe the origin (\emph{left}), then compute the values along the two coordinate axes (\emph{center}) before consecutively determining the remaining values (\emph{right}).}
  \label{fig:qnet}
\end{figure}

More concretely, in case of $d = 2$:
\begin{itemize}
\item Select $p_{0,0}$ subject to $s(p_{0,0},p_{0,0}) = 0$ (Figure~\ref{fig:qnet}, left).
\item For $i \ge 1$, select $p_{i,0}$ subject to $s(p_{i-1,0},p_{i,0}) = s(p_{i,0},p_{i,0}) = 0$ and similar for $p_{-i,0}$, $p_{0,i}$, and $p_{0,-i}$ (Figure~\ref{fig:qnet}, middle).
\item Select $p_{1,1}$ subject to $s(p_{1,0},p_{1,1}) = s(p_{0,1},p_{1,1}) = s(p_{1,1},p_{1,1}) = 0$ and similar for $p_{-1,1}$, $p_{1,-1}$, and $p_{-1,-1}$. In the same manner, we compute $p_{i,j}$ for $i$, $j > 0$ once $p_{i-1,j}$ and $p_{i,j-1}$ are determined (and similar for $p_{-i,j}$, $p_{i,-j}$, and $p_{-i,-j}$) (Figure~\ref{fig:qnet}, right).
\end{itemize}

Let us conclude this section by emphasizing that the outlined construction shows \emph{theoretical existence} of snapping four-bar nets. It suffers from similar drawbacks as the constructions described in Section~\ref{sec:construction}: The actual position of revolute axes is hard to control as is the revolute angles between two configurations.

\section{Surface Design}
\label{sec:surface-design}

Having established existence of snapping four-bar nets, we now turn to their actual construction. In order to obtain feasible mechanisms, it is necessary to control the axes' positions as well as the \emph{snap angles}, that is, the revolute angles in each axis between the two configurations. The former is important because the axes need to be connected by rigid links of reasonable size and shape, the latter is important for two reasons:
\begin{enumerate}
\item Large snap angles tend to come along with much strain on the links and an increased risk of self collision during the snap.
\item Small snap angles will produce structures that are close to being shaky with an entirely different mechanical behavior.
\end{enumerate}

The purpose of this section is to present a method for generating ``surface-like'' snapping four-bar nets while allowing control over the issues mentioned above. The approach builds on the interpretation of snapping four-bar nets as $\SQ$-nets in the Study quadric and their associated discrete rolling motions. A key observation is that these discrete rolling motions can themselves be generated from a pair of discrete surfaces with congruent faces, thereby connecting our problem with a further class of well-studied objects in discrete differential geometry.

For dimension $d=2$, we propose a procedure that begins with a special infinitesimally flexible discrete surface, realized as a discrete map $f\colon \mathbb{Z}^2 \to \mathbb{R}^3$. Via Whiteley de-averaging \cite{Whiteley2004}, we then construct two discrete surfaces with congruent faces, which in turn generate a discrete rolling motion and thus a snapping four-bar net.

\subsection{Rolling of Discrete Surfaces with Congruent Faces}
\label{sec:rolling-nets}

The basic idea is to construct discrete rolling motions not as quad nets in the Study quadric $\SQ$ but via the ``rolling'' of one discrete surface in $\mathbb{R}^3$ onto a second isometric surface.

Thus, assume that $\Ng$, $\Nh \colon \mathbb{Z}^2 \to \mathbb{R}^3$ are two discrete surfaces with directly congruent faces. In other words, for any quadruple $(\vvi,\vvj,\vvk,\vvl)$ of multi-indices describing a face (that is, $\vvi=(m,n)$, $\vvj=(m+1,n)$, $\vvk=(m+1,n+1)$, $\vvl=(m,n+1)$ for some $m$, $n \in \mathbb{Z}$), the two spatial quadrilaterals
\[
  \Ng_\vvi, \Ng_\vvj, \Ng_\vvk, \Ng_\vvl
  \quad\text{and}\quad
  \Nh_\vvi, \Nh_\vvj, \Nh_\vvk, \Nh_\vvl
\]
correspond in a rigid body transformation. We now assign a checkerboard black-and-white coloring to the faces of both discrete surfaces, so that we can distinguish pairs of congruent black faces from pairs of congruent white faces (Figure~\ref{fig:rolling-surfaces}, top left and top center). Fixing the discrete surface~$\Ng$ in space, we then ``roll'' the discrete surface~$\Nh$ on $\Ng$ to generate rotations that constitute a rotation net.

To this end, consider the four poses of the net $\Nh$ for which the corresponding black faces of $\Ng$ and $\Nh$ around one given white face of~$\Ng$ coincide (Figure~\ref{fig:rolling-surfaces}, bottom). Any two neighboring poses are related by a rotation around an axis passing through a vertex of~$\Ng$; this follows from Euler's Rotation Theorem \cite[Chapter~VII]{bottema90}. Moreover, by construction, the composition of these four rotations is the identity.

Thus, by our results in Section~\ref{sec:rotation-quadrilaterals}, \emph{rolling $\Nh$ around one white face of~$\Ng$}, gives rise to a rotation quadrilateral and, subsequently, to one single snapping four-bar. In particular, it generates one face of an $\SQ$-net~$p\colon \mathbb{Z}^2 \to \SQ$. The vertices of the corresponding face in the $\SQ$-net are naturally associated to the black faces of the checkerboard pattern (Figure~\ref{fig:rolling-surfaces}, top right).

This concept of rolling around a single face naturally extends to a discrete rolling of the entire discrete surface $\Nh$ on $\Ng$. In this way, two discrete surfaces $\Ng$ and $\Nh$ with directly congruent faces give rise, via the checkerboard pattern, to an $\SQ$-net in the Study quadric~$\SQ$.

The immediate advantage of this construction is that each revolute axis in the fixed frame (the frame of $\Ng$) contains a unique vertex of $\Ng$, and each revolute axis in the moving frame (the frame of $\Nh$) contains a unique vertex of $\Nh$. Hence, the geometry of the two discrete surfaces is directly related to the shape of the rigid links required to realize the snapping four-bar net as a mechanism.

At this stage, we do not control the axis directions, but Equation~\eqref{eq:axis-direction} in Section~\ref{sec:axis-angles} will provide a simple formula in a particular case. However, this is less relevant as only a relatively short portion of each revolute axis is needed for a mechanical revolute joint to be attached.

\begin{figure}
  \centering
  \includegraphics[]{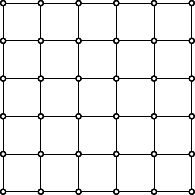}\qquad
  \includegraphics[]{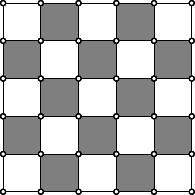}\qquad
  \includegraphics[]{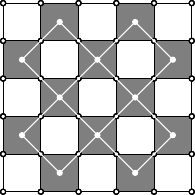}\\[1ex]
  \includegraphics{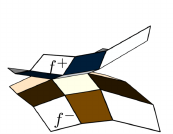}%
  \includegraphics{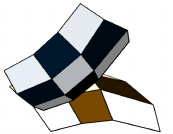}%
  \includegraphics{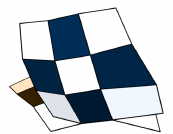}%
  \includegraphics{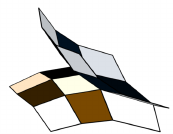}
  \caption{Top: A graph with $\mathbb{Z}^2$-combinatorics with black and white coloring of the faces and the related dual graph. Bottom: The quad net $\Nh$ rolls on the quad net $\Ng$. Both $\Ng$ and $\Nh$ correspond to the original graph while the discrete rolling motion corresponds to the dual graph (\emph{right}).}
  \label{fig:rolling-surfaces}
\end{figure}

\begin{remark}
  The discrete surfaces~$\Ng$ and $\Nh$ might be considered as fixed and moving axode, respectively, of a discrete 2-dimensional motion where neighboring poses correspond in a rotation. This suggests viewing our proposed approach as a discrete version of a classical concept in differential geometry, the rolling of two isometric smooth surfaces without slipping. In particular, the resulting rotation nets can be interpreted as a discrete counterpart of the ``Drehriss'' discussed in~\cite{blaschke13}. Our approach also aligns a lot with concepts of theoretical space kinematics like rolling and slipping of axodes. Note however, that there is no slipping in our case (compare Figures~\ref{fig:rotations} and \ref{fig:rolling-surfaces}). Moreover, two-parametric motions are rare in classic kinematics literature.
\end{remark}

The attentive reader might have noticed that our construction requires direct congruence only between corresponding \emph{black} faces of $\Ng$ and $\Nh$. However, as we will show below, using discrete surfaces with \emph{all} faces directly congruent does not impose a significant restriction: For a given snapping four-bar net, the discrete surfaces $\Ng$ and $\Nh$ can be chosen such that all corresponding faces are congruent. We begin by considering the case of a single snapping four-bar.

\begin{proposition}
  \label{prop:congr_quad}
  For every snapping four-bar with axes~$(R_\vvi, R_\vvj, R_\vvk, R_\vvl)$, there exists a quadrilateral with vertices~$(\Ng_\vvi, \Ng_\vvj, \Ng_\vvk, \Ng_\vvl)$ located on the respective axes that is congruent to the corresponding quadrilateral~$(\Nh_\vvi, \Nh_\vvj, \Nh_\vvk, \Nh_\vvl)$ in the second configuration~$(R'_\vvi, R'_\vvj, R'_\vvk, R'_\vvl)$ of the snapping four-bar.

  Generically, two neighboring points, say $\Ng_\vvi$ and $\Ng_\vvj$, can be prescribed whence there exists a $2$-parameter family of solutions.
\end{proposition}

One may view Proposition~\ref{prop:congr_quad} as the converse of \cite[Theorem~6(i)]{stachel25} where existence of axes $(R_\vvi, R_\vvj, R_\vvk, R_\vvl)$ to given congruent quadrilaterals $(\Ng_\vvi,\Ng_\vvj,\Ng_\vvk,\Ng_\vvl)$ and $(\Nh_\vvi,\Nh_\vvj,\Nh_\vvk,\Nh_\vvl)$ is proved.

\begin{proof}[Proof of Proposition~\ref{prop:congr_quad}]
  Let $(\NTg_\vvi, \NTg_\vvj, \NTg_\vvk, \NTg_\vvl)$ be an arbitrary quadrilateral whose vertices lie on the axes of a snapping four-bar in one configuration and denote by $(\NTh_\vvi, \NTh_\vvj, \NTh_\vvk, \NTh_\vvl)$ the corresponding quadrilateral in the second configuration. We show that, by appropriately translating the two vertices $\NTg_\vvk$ and $\NTg_\vvl$ on neighboring axes along their respective axes, one can obtain a modified quadrilateral such that the new corresponding quadrilateral in the second configuration is congruent to it.

To this end, we assume that the second configuration of the snapping four-bar has been rigidly transformed into a pose where the axes $R_\vvk$ and $R_\vvl$ coincide with $R'_\vvk$ and $R'_\vvl$, respectively. We then consider the perpendicular bisector planes~$\mathcal{E}_\vvi$ of $\NTg_\vvi$ and $\NTh_\vvi$, as well as $\mathcal{E}_\vvj$ of $\NTg_\vvj$ and $\NTh_\vvj$.

The new vertices are defined by intersecting the coincident axes $R_\vvk=R'_\vvk$ and $R_\vvl=R'_\vvl$ with the bisector planes corresponding to the diagonally opposite vertices:
\begin{equation*}
\Ng_\vvk = \Nh_\vvk \coloneqq \mathcal{E}_i \cap R_k \quad \text{and} \quad \Ng_\vvl = \Nh_\vvl \coloneqq \mathcal{E}_j \cap R_l.
\end{equation*}
In this way, we obtain two new quadrilaterals $Q\coloneqq(\NTg_\vvi, \NTg_\vvj, \Ng_\vvk, \Ng_\vvl)$ and $Q'\coloneqq(\NTh_\vvi, \NTh_\vvj, \Nh_\vvk, \Nh_\vvl)$. By construction, the triangles $(\NTg_\vvj, \Ng_\vvk, \Ng_\vvl)$ and $(\NTh_\vvj, \Nh_\vvk, \Nh_\vvl)$ are congruent, because they are mirror images of each other with respect to the bisector plane~$\mathcal{E}_j$. The same holds true for the triangles $(\Ng_\vvk, \Ng_\vvl, \NTg_\vvi)$ and $(\Nh_\vvk, \Nh_\vvl, \NTh_\vvi)$ with respect to~$\mathcal{E}_i$. Because the distance from $\NTg_\vvi$ to $\NTg_\vvj$ equals the distance from $\NTh_\vvi$ to $\NTh_\vvj$, the two newly constructed quadrilaterals $Q$ and $Q'$ are congruent to each other.
\end{proof}

\begin{remark}
  \label{rem:congruent-quads}
  We do not claim or prove that the two quadrilaterals $(\Ng_\vvi, \Ng_\vvj, \Ng_\vvk, \Ng_\vvl)$ and $(\Nh_\vvi, \Nh_\vvj, \Nh_\vvk, \Nh_\vvl)$ in Proposition~\ref{prop:congr_quad} are \emph{directly} congruent. However, we believe that, generically, directly congruent quads exist, provided that all four vertices are allowed to be shifted appropriately along the rotation axes.

  As we will see in the next theorem, the two-parameter freedom stated in the proposition is necessary for a snapping four-bar net to avoid additional global compatibility conditions, which appear to be difficult to control.
\end{remark}

\begin{theorem}
  \label{th:1}
  For every snapping four-bar net there exist discrete surfaces~$\Ng$, $\Nh \colon \mathbb{Z}^2 \to \mathbb{R}^3$ and checkerboard colorings of $\Ng$ and $\Nh$ such that
  \begin{itemize}
  \item corresponding black faces of $\Ng$ and $\Nh$ are directly congruent,
  \item corresponding white faces of $\Ng$ and $\Nh$ are congruent, and
  \item the discrete rolling of the black faces of~$\Nh$ onto the black faces of~$\Ng$ generates the rotation net of the snapping four-bar net.
  \end{itemize}
\end{theorem}

\begin{proof}
  We provide an iterative construction of $\Ng$. Once the values of $\Ng$ are given, the discrete surface~$\Nh$ is automatically defined.

  The snapping four-bar net provides revolute axes in the fixed and in the moving frame and we arbitrarily prescribe the values of $\Ng$
  \begin{itemize}
  \item on the central black quad with vertices $(0,0)$, $(1, 0)$, $(1,1)$, and $(0,1)$ and
  \item on the remaining vertices along the diagonal lines given by $n-m=0$ and $n+m=+1$, respectively.
  \end{itemize}
  This initial data is illustrated in Figure~\ref{fig:th:1} with filled dots.

  Proposition~\ref{prop:congr_quad} now provides us with the values of $\Ng$ on the four white quads around the central quad (empty dots in Figure~\ref{fig:th:1}). Together with the prescribed diagonal values, all vertices in the one-ring neighborhood of the central quad are determined. Proceeding in like manner we determine the values of $\Ng$ iteratively on the two-ring neighborhood, the three-ring neighborhood etc.\ of the central quad.
\end{proof}

\begin{figure}
  \centering
  \includegraphics[]{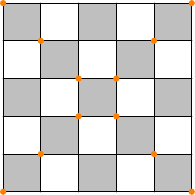}\quad
  \includegraphics[]{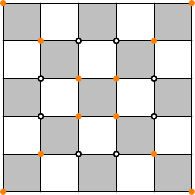}\quad
  \includegraphics[]{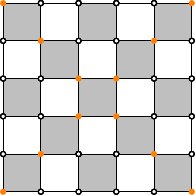}\quad
  \caption{Initial data and iterative construction for Theorem~\ref{th:1}.}
  \label{fig:th:1}
\end{figure}

\begin{remark}
  Similar to Remark~\ref{rem:congruent-quads}, Theorem~\ref{th:1} does not guarantee direct congruence of the white quads. At present, it remains unclear whether, for every snapping four-bar net, there exist discrete surfaces $\Ng$ and $\Nh$ whose black and white faces are directly congruent.
\end{remark}

We end this subsection with an observation showing that the geometry of a snapping 4-bar is strongly related to the geometry of the associated congruent quads discussed in the previous paragraphs.

\begin{proposition}\label{prop:planarity}
  Suppose that a snapping four-bar admits a \emph{planar} quadrilateral with vertices on its axes that is congruent in the second configuration. Then any two neighboring rotation axes of the four-bar are coplanar.
\end{proposition}

\begin{proof}
  With the notation in Proposition~\ref{prop:congr_quad}, we assume that the second configuration of the snapping four-bar has been rigidly transformed such that the axes $R_k$ and $R_l$ coincide with $R'_k$ and $R'_l$, respectively. Since the two associated quadrilaterals~$(\Ng_\vvi, \Ng_\vvj, \Ng_\vvk, \Ng_\vvl)$ and $(\Nh_\vvi, \Nh_\vvj, \Nh_\vvk, \Nh_\vvl)$ are planar, the two bisecting planes $\mathcal{E}_i$ and $\mathcal{E}_j$ of diagonally opposite vertices coincide, that is, $\mathcal{E}_i=\mathcal{E}_j$. On the other hand, for every snapping four-bar we know that $R_\vvk \in \mathcal{E}_\vvi$ and $R_\vvl \in \mathcal{E}_\vvj$. Consequently, any two neighboring axes are coplanar.
\end{proof}

Proposition~\ref{prop:planarity} extends naturally to snapping four-bar nets. Indeed, any pair of face-isometric surfaces with planar faces gives rise to a snapping four-bar net composed of intersecting neighboring axes.

Theorem~\ref{th:1} and the other insights of this section pose the question, how to find pairs $\Ng$, $\Nh$ of quad nets in~$\mathbb{R}^3$ with directly congruent faces? In addition, it is also important to control some sort of ``distance'' between $\Ng$ and $\Nh$: If $\Ng$ and $\Nh$ are too far apart, the snap might be too large. If $\Ng$ and $\Nh$ are too close the snap becomes weak and rather feels like one shaky surface. In the upcoming Sections~\ref{sec:isometric-from-infinitesimally} and \ref{sec:koenigs-nets} we will address this problem.

\subsection{Isometric Quad Nets from Infinitesimally Flexible Nets}
\label{sec:isometric-from-infinitesimally}

Going back to R.~Sauer \cite{sauer33, sauer37}, we define:
\begin{definition}
  Two discrete surfaces are called
  \begin{itemize}
  \item \emph{isometric} if corresponding edges have the same length;
  \item \emph{face isometric} if corresponding faces are directly congruent;
  \item \emph{star isometric} if corresponding vertex-stars (the geometric figures consisting of all edges emanating from a fixed vertex) are directly congruent.
  \end{itemize}
  Note that face/star isometric nets are isometric as well.
\end{definition}

\begin{definition}\label{def:IID}
  A discrete surface~$f\colon \mathbb{Z}^2 \to \mathbb{R}^3$ is \emph{infinitesimally flexible (shaky)} if it admits a nontrivial velocity field~$q\colon \mathbb{Z}^2 \to \mathbb{R}^3$ such that corresponding edges are orthogonal, that is, for any edge~$(\vvi,\vvj)$, we have
  \begin{equation*}
    \langle f_\vvi - f_\vvj, q_\vvi - q_\vvj \rangle = 0.
  \end{equation*}
  In this case we call $q$ an \emph{infinitesimal isometric deformation} (an \emph{IID} for short) and $(f,q)$ an \emph{IID-pair.}
\end{definition}

The velocity field $q$ is called \emph{trivial} if it is obtained as velocity vector field of $f$ undergoing a rigid body motion. We also note that the velocity field $q$ is never unique. If $(f,q)$ is an IID-pair, then so are $(f, tq)$ and $(f,q+v)$ for any constant $t \in \mathbb{R}\setminus\{0\}$ and $v \in \mathbb{R}^3$. The same is true for $(f,q+r)$ where $r$ is a trivial IID of $f$.

An example of an infinitesimally flexible discrete surface~$f$ with IID~$q$ is depicted in the top right of Figure~\ref{fig:enneper}. The net~$f$ is the discrete Enneper surface from~\cite{Pinkall1996}, a discrete counterpart of the classical Enneper minimal surface. We will explain its construction later in Section~\ref{sec:koenigs-nets}.

Definition~\ref{def:IID} captures the intuitive notion that the lengths of edges of the discrete surface~$f$ in the deformation
\begin{equation}
  \label{eq:deformed-net}
  t \in \mathbb{R} \mapsto f + tq
\end{equation}
are preserved in first order: For any edge $(\vvi,\vvj)$ denote by $L_{\vvi,\vvj}(t)$ the squared edge length of the deformed net \eqref{eq:deformed-net}. We find
\begin{equation*}
  \begin{aligned}
    L_{\vvi,\vvj}(t) &= \langle(f_\vvi+tq_\vvi)-(f_\vvj+tq_\vvj), (f_\vvi+tq_\vvi)-(f_\vvj+tq_\vvj)\rangle \\
              &= \langle f_\vvi-f_\vvj, f_\vvi-f_\vvj \rangle + 2t \langle f_\vvi-f_\vvj, q_\vvi-q_\vvj \rangle + t^2 \langle q_\vvi-q_\vvj, q_\vvi-q_\vvj \rangle
  \end{aligned}
\end{equation*}
and indeed obtain $\od{L_{\vvi,\vvj}}{t}(0) = 0$ if and only if $\langle f_\vvi-f_\vvj, q_\vvi-q_\vvj \rangle = 0$.

While the edge lengths in the deformation~\eqref{eq:deformed-net} change, rather surprisingly, for any choice of $t \in \mathbb{R} \setminus \{0\}$ the two discrete surfaces
\begin{equation*}
  f^{\pm t} \coloneqq f \pm t q
\end{equation*}
have equal edge lengths and are therefore isometric. This method of generating pairs of isometric discrete surfaces from an infinitesimally flexible discrete surface is called \emph{Whiteley de-averaging}~\cite{Whiteley2004}.

Equality of corresponding edge lengths of $f^{+t}$ and $f^{-t}$ follows directly from the following computation:
\begin{equation*}
  \begin{aligned}
    \langle f^{+t}_\vvi-f^{+t}_\vvj, f^{+t}_\vvi-f^{+t}_\vvj \rangle &= \langle f_\vvi-f_\vvj, f_\vvi-f_\vvj \rangle + t^2 \langle q_\vvi-q_\vvj, q_\vvi-q_\vvj \rangle \\
    & = \langle  f^{-t}_\vvi-f^{-t}_\vvj, f^{-t}_\vvi-f^{-t}_\vvj \rangle.
  \end{aligned}
\end{equation*}
The converse construction produces infinitesimally flexible structures from isometric ones and is called \emph{Whiteley averaging} \cite{Whiteley2004}: Given an isometric pair $f^{+t}$ and $f^{-t}$, a shaky quad net can be recovered as $f = \frac{1}{2}(f^{+t}+f^{-t})$. Its IID is $q \coloneqq \frac{1}{2}(f^{+t}-f^{-t})$.

In general, the isometric pair $f^{+t}$, $f^{-t}$ is neither face isometric nor star isometric. However, there exist specific IID-pairs $(f,q)$ giving rise to face isometric or star isometric $f^+$, $f^-$~\cite{sauer33}. In this case the deformation $t \mapsto f + tq$ is called \emph{face rigid} or \emph{star rigid,} respectively.

\begin{figure}
  \centering
  \includegraphics[]{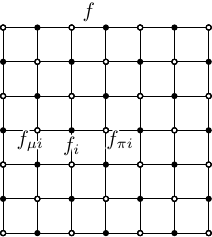}\quad
  \includegraphics[]{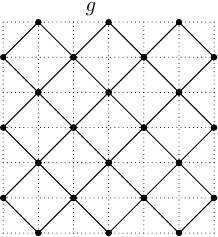}\quad
  \includegraphics[]{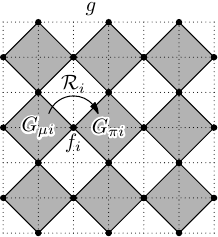}
  \caption{By introducing a black-and-white coloring on vertices and switching to the diagonal net, we turn the vertex isometric quad net $f$ (\emph{left}) into the face isometric quad net $g$ (\emph{center}). A checkerboard coloring of $g$ gives rise to a discrete rolling motion and hence an $\SQ$-net in the Study quadric. Note that we omit the superscript ``$\pm t$'' in all labels because the drawings refer to both nets.}
  \label{fig:graphs}
\end{figure}

In principle, we are interested in the face rigid case. However, the diagonal nets of star isometric nets are face rigid:
  \begin{itemize}
\item Color the vertices of $f^{\pm t}$ in black and white according to their parity, c.f. Figure~\ref{fig:graphs}, left. (The parity of the vertex $f_{m,n}$ is even if $m+n$ is even and it is odd if $m+n$ is odd.)
\item Construct new discrete surfaces~$g^{\pm t}$ by connecting any two neighboring (black) vertices of a vertex-star with white center by an edge (Figure~\ref{fig:graphs}, center).
\item Color the faces of~$g^{\pm t}$ according to a checkerboard pattern (Figure~\ref{fig:graphs}, right).
  \end{itemize}
  Thus, we can still use star rigid nets for our purposes after a change of combinatorics. This is helpful as the theory of star rigid deformations is better developed and contains interesting surface classes such as discrete minimal surfaces or CMC surfaces~\cite{bobenko08, Pinkall1996, bobenko08b, pinkallcmc}.

\begin{figure}
  \centering
  \includegraphics{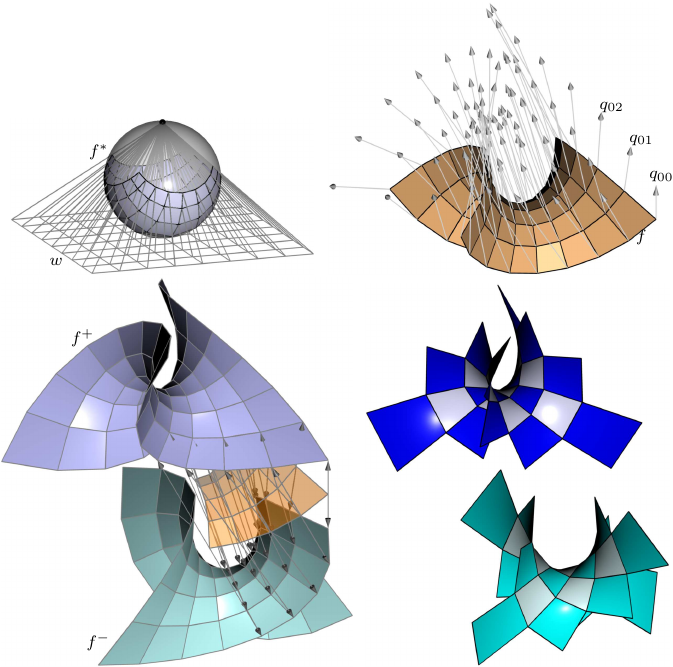}
  \caption{\emph{Top left}: The Koenigs dual $f^*$ is obtained via inverse stereographic projection. \emph{Top right}: The Koenigs net $f$ is infinitesimally flexible with IID $q$. \emph{Bottom left}: The discrete quad nets $f^+$ and $f^-$ are obtained by adding and subtracting $f$ and $q$. A portion of $f$ is depicted as well. \emph{Bottom right}: The diagonal nets of $f^+$ and $f^-$ have congruent (non-planar) faces. The checkerboard patterns define a discrete rolling motion and hence a snapping quad net of snapping four-bars.}
  \label{fig:enneper}
\end{figure}

Figure~\ref{fig:enneper} displays the nets $f^+ \coloneqq f^{+t}\mid_{t=1}$, $f^- \coloneqq f^{-t}\mid_{t=-1}$ and $f$ in the bottom left corner. They are obtained by adding and subtracting the IID~$q$ of the discrete Enneper surface $f$. Corresponding edges have indeed equal lengths. The deformation induced by the IID~$q$ is vertex rigid but not face rigid. Therefore, corresponding faces of $f^+$ and $f^-$ are not directly congruent but corresponding vertex stars are. Hence, their respective diagonal nets, displayed in Figure~\ref{fig:enneper}, bottom right, are face isometric.

The advantage of obtaining isometric discrete surface pairs for discrete rolling and subsequent building of snapping four-bar nets via Whiteley de-averaging is threefold:
\begin{itemize}
\item Infinitesimally flexible quad nets in $\mathbb{R}^3$ are well-studied in literature, see for example~\cite{bobenko08, sauer33, sauer53, schief08}.
\item Whiteley de-averaging is a straightforward and simple procedure. It depends on a real parameter $t$ that allows to control the ``distance'' of $f^+$ and $f^-$ to the original shaky quad net $f$ and also between each other.
\item For a special case (discrete Koenigs nets) we can provide simple formulas for snap angles and also axis directions, c.f. Section~\ref{sec:koenigs-nets} and in particular Equations~\eqref{eq:revolute-angle} and \eqref{eq:axis-direction}, that are directly related to the geometry of the discrete surfaces involved.
\end{itemize}

\subsection{Discrete Koenigs Nets}
\label{sec:koenigs-nets}

To complete our pipeline for constructing snapping four-bar nets, we describe the class of discrete surfaces that admit a star-rigid infinitesimal isometric deformation. This class includes the discrete Enneper surface of Figure~\ref{fig:enneper}, as well as various other well-studied surface classes, such as discrete surfaces with constant mean curvature. Moreover, it provides good control over the snap angles.

\begin{definition}[\cite{bobenko08,sauer33}]\label{def:koenigs}
  A discrete quad surface~$f\colon \mathbb{Z}^2 \to \mathbb{R}^3$ with planar faces is called a \emph{Koenigs net} if there exists a discrete surface~$f^*$ (the \emph{Koenigs dual} of $f$) such that
\begin{itemize}
\item[(K1)] $f_\vvi-f_\vvj \parallel f^*_\vvi-f^*_\vvj$ for every edge $(\vvi,\vvj)$
\item[(K2)] $f_\vvi-f_\vvk \parallel f^*_\vvj-f^*_\vvl$ for any two diagonals $(\vvi,\vvk)$ and $(\vvj,\vvl)$ of a face.
\end{itemize}
\end{definition}
In other words, the quad net $f$ is Koenigs if corresponding faces of $f$ and $f^*$ have corresponding parallel edges and parallel non-corresponding diagonals.

\begin{remark}
  \label{rem:koenigs-construction}
  We remark that, due to the symmetries in (K1) and (K2), the Koenigs dual is again a Koenigs net. In particular, given a Koenigs net $f^\star$, the net $f$ can be uniquely reconstructed from two initial values $f_\vvi$ and $f_\vvj$ (such that $f_\vvi-f_\vvj \parallel f^*_\vvi-f^*_\vvj$) on one edge $(\vvi,\vvj)$ because
  \begin{itemize}
  \item all edge directions and diagonal directions are known and
  \item we can iteratively construct $f$ by intersecting straight lines in space.
  \end{itemize}
  In doing so, the different ways of computing one particular vertex all yield the same result so that one does not run into contradictions \cite[Section~2.3]{bobenko08}.
\end{remark}

\begin{proposition}
  \label{prop:koenigs_inf}
  Any Koenigs net~$f$ with Koenigs dual~$f^*$ is infinitesimally flexible. A velocity vector field $q$ is uniquely determined by one initial value and the formula
  \begin{equation}
    \label{eq:koenigs_formula}
    q_\vvj - q_\vvi = f_\vvi^* \times (f_\vvj-f_\vvi)= f_\vvj^* \times (f_\vvj - f_\vvi)
  \end{equation}
  for any edge~$(\vvi,\vvj)$.
\end{proposition}

Proposition~\ref{prop:koenigs_inf} is due to \cite{sauer33}. We provide a proof in English language and modern notation.

\begin{proof}[Proof of Proposition~\ref{prop:koenigs_inf}]
  We first obverse that, due to (K1), the right equality in \eqref{eq:koenigs_formula} indeed holds:
  \begin{equation*}
    f_\vvj^* \times (f_\vvj-f_\vvi) - f_\vvi^* \times (f_\vvj - f_\vvi) = (f_\vvj^* - f_\vvi^*) \times (f_\vvj-f_\vvi) = 0.
  \end{equation*}
  The vector field $q$ as described above is unique because \eqref{eq:koenigs_formula} allows to construct $q_\vvj$ from given $q_\vvi$ for any edge $(\vvi,\vvj)$. In order to show existence, we need to prove that for any face $(\vvi,\vvj,\vvk,\vvl)$ with $q_\vvi$ given, the construction of $q_\vvk$ via $q_\vvj$ and of $q'_\vvk$ via $q_\vvl$ gives the same vertex $q_\vvk = q'_\vvk$. Indeed,
  \begin{align*}
    q_\vvk &= q_\vvj + f_\vvj^* \times(f_\vvk - f_\vvj) = q_\vvi + f_\vvj^* \times(f_\vvj - f_\vvi) + f_\vvj^* \times (f_\vvk - f_\vvj), \\
    q'_\vvk &= q_\vvl + f_\vvl^* \times(f_\vvk - f_\vvl) = q_\vvi + f_\vvl^* \times(f_\vvl - f_\vvi) + f_\vvl^* \times (f_\vvk - f_\vvl).
  \end{align*}
  By (K2), it directly follows that
  \begin{multline*}
    q_\vvk - q'_\vvk =
    f_\vvj^* \times(f_\vvj - f_\vvi) + f_\vvj^* \times (f_\vvk - f_\vvj) -
    f_\vvl^* \times(f_\vvl - f_\vvi) - f_\vvl^* \times (f_\vvk - f_\vvl) \\=
    f_\vvj^* \times (f_\vvk - f_\vvi) - f_\vvl^* \times (f_\vvk - f_\vvi) =
    (f_\vvj^* - f_\vvl^*) \times (f_\vvk - f_\vvi) = 0
  \end{multline*}
  and the vector field~$q$ is indeed well-defined.

  Finally, taking into account (K1), we conclude that $q$ is a velocity field that defines an infinitesimally isometric deformation for $f$:
  \begin{equation*}
    \langle f_\vvj - f_\vvi, q_\vvj - q_\vvi \rangle =
    \langle f_\vvj - f_\vvi, f_\vvi^* \times (f_\vvj - f_\vvi) \rangle =
    \langle f_\vvi^*,  (f_\vvj - f_\vvi)\times (f_\vvj - f_\vvi) \rangle =
    0.\qedhere
  \end{equation*}
\end{proof}

The simplest non-trivial example of a minimal Koenigs net is certainly the discrete analog of Enneper's minimal surface of \cite{Pinkall1996}. Its Koenigs dual is given by its discrete Gauss map; hence, its vertices lie on a sphere.

\begin{example}[Discrete Enneper Minimal Surface \cite{Pinkall1996}]
  \label{ex:enneper_surface}
  Let
  \begin{equation*}
    w\colon \mathbb{Z}^2 \to \mathbb{R}^3,\quad
    (m,n) \to (m,n,0)
  \end{equation*}
  be the regular $\mathbb{Z}^2$-grid, embedded into $\mathbb{R}^3$ as plane with equation $z = 0$, and denote the inverse stereographic projection of~$w$ from the center $(0,0,4)$ onto the sphere with center $(0,0,2)$ and radius $2$ by~$f^*$. We have
  \[
    f^*_{m,n} = \frac{1}{m^2+n^2+16}(16m, 16n, 4m^2+4n^2).
  \]
  We can compute the dual surface~$f$, the discrete Enneper minimal surface of Figure~\ref{fig:enneper}, top right, as in Remark~\ref{rem:koenigs-construction}. A velocity field~$q$ is determined by \eqref{eq:koenigs_formula}. For some initial value $q_{00}$, we obtain for example
  \[
    q_{01} = q_{00} + \frac{1}{123}(16,112,64),\quad
    q_{02} = q_{00} + \frac{1}{123}(56,64,128),\quad\text{etc.}
  \]
  In order to produce the images in Figure~\ref{fig:enneper} we used $q_{00} = (0,0,1)$. This choice is inconsequential as it only effects a constant translation of the discrete surfaces $f^+$ and~$f^-$.

A discrete Enneper surface was also used to generate the data of the physical prototype of a snapping four-bar net shown in Figure~\ref{fig:prints} (see Subsection~\ref{sec:3d-print} for more details).
\end{example}

\subsection{Revolute Axes and Snap Angles From Koenigs Nets}
\label{sec:axis-angles}

Using Koenigs nets to initiate our construction has one big advantage. \emph{It allows for a simple formula for the snap angles.} To describe it, we proceed by introducing some notation. Let
\begin{equation*}
  \crossM{v} \coloneqq
  \begin{pmatrix} 0 & -v_3 & v_2 \\ v_3 & 0 & -v_1\\ -v_2 & v_1 & 0 \end{pmatrix}
\end{equation*}
be the skew-symmetric matrix associated to the vector $v=(v_1,v_2,v_3) \in \mathbb{R}^3$ such that $\crossM{v}x=v \times x$ for any $x \in \mathbb{R}^3$. Moreover, we denote the \emph{Cayley transform} of the vector~$v$ by
\begin{equation*}
    \cay{v}\coloneqq(I-\crossM{v})(I+\crossM{v})^{-1},
\end{equation*}
where $I$ denotes the identity matrix, c.f. \cite[Section~B.4]{lynch2017}. The matrix~$\cay{v}$ is then an orthogonal matrix which defines a proper rotation in~$\mathbb{R}^3$. Note that the inverse of a Cayley transform is itself a Cayley transform, given by the formula $C(v)^{-1} = C(-v)$.

\begin{figure}
  \centering
  \includegraphics{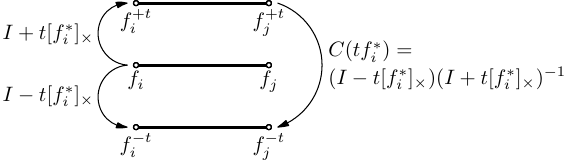}
  \caption{Illustration for the derivation of Equation~\eqref{eq:edgeRelation}.}
  \label{fig:cayley}
\end{figure}

We have the following relation between an edge of a deformed net~$f^{+t}\coloneqq
f+tq$ and an edge $(\vvi,\vvj)$ of the shaky net~$f$:
\begin{align*}
    (f^{+t}_\vvj - f^{+t}_\vvi) &= (f_\vvj - f_\vvi) + t (q_\vvj - q_\vvi) \\
    &= (f_\vvj - f_\vvi) + t \big( f_\vvi^\star \times (f_\vvj - f_\vvi) \big)\\
    &= (I + t\crossM{f^*_\vvi})(f_\vvj - f_\vvi).
\end{align*}
This is illustrated in Figure~\ref{fig:cayley}. Consequently, we obtain that corresponding edges of the discrete surfaces $f^{+t}$ and $f^{-t}$ are related by Cayley transforms which are given explicitly by $f^*$ and the deformation parameter~$t$,
\begin{equation}
  \begin{aligned}
    \label{eq:edgeRelation}
    ( f^{-t}_\vvj - f^{-t}_\vvi ) &= (I-t\crossM{f^*_\vvi})(I+t\crossM{f^*_\vvi})^{-1} (f^{+t}_\vvj - f^{+t}_\vvi) \\
    &= C(t f_\vvi^\star) (f^{+t}_\vvj - f^{+t}_\vvi)
  \end{aligned}
\end{equation}
(c.f. Figure~\ref{fig:cayley} as well). We observe that Equation~\eqref{eq:edgeRelation} holds for any neighboring point~$f_\vvj^{\pm t}$ of $f_\vvi^{\pm t}$. Thus, the rotation $ C(t f_\vvi^\star)$ is the same for any edge of the vertex star of~$f_\vvi^{+t}$ and we conclude that corresponding vertex stars of $f^{+t}$ and $f^{-t}$ are directly congruent.

Therefore, we need to pass on to the diagonal nets $g^{_\pm t}$ as described in Section~\ref{sec:rolling-nets} and illustrated in Figure~\ref{fig:graphs}. We denote the faces of~$g^{\pm t}$ by $G^{\pm t}$, where $G_{\vvi}^{\pm t}$ denotes the pair of faces obtained from the vertex stars at~$f_{\vvi}^{\pm t}$. From this new surface pair~$g^{\pm t}$, we then obtain a rotation net via rolling $g^{+t}$ onto~$g^{-t}$ along black faces as described in Subsection~\ref{sec:rolling-nets}.

We are now going to explicitly describe the relative rotations between adjacent black faces whose indices differ in one coordinate direction, be it the first or the second. Denote by $\pi\colon \mathbb{Z}^2 \to \mathbb{Z}^2$ the operator that increments the index by one in the relevant coordinate direction and set $\mu\coloneqq\pi^{-1}$ (the decrement operator). In view of
\begin{equation*}
  ( f_{\pi\vvi}^{-t} - f^{-t}_{\vvi} ) =  C(t f_{\pi\vvi}^\star)  C(t f_{\mu\vvi}^\star)^{-1} ( C(t f_{\mu\vvi}^\star)(f_{\pi\vvi}^{+t} - f^{+t}_{\vvi}) ),
\end{equation*}
we define
\begin{equation*}
  \mathcal{R}^t_\vvi \coloneqq  C(t f_{\pi\vvi}^\star)  C(t f_{\mu\vvi}^\star)^{-1} = C(t f_{\pi\vvi}^\star)  C(-t f_{\mu\vvi}^\star).
\end{equation*}
The matrix~$\mathcal{R}_\vvi^t$ is then a rotation matrix. It describes the relative rotation between the two poses of~$g^{+t}$ where the face pairs $G^{-t}_{\mu\vvi}$, $G^{+t}_{\mu\vvi}$ and $G^{-t}_{\pi\vvi}$, $G^{+t}_{\pi\vvi}$, respectively, are aligned (Figure~\ref{fig:graphs}, right). More precisely, the matrices~$\mathcal{R}_\vvi^t$ describe the discrete rolling motion when rolling $g^{+t}$ onto $g^{-t}$ along black faces. Thus, this provides a way to describe rotation nets of snapping four-bar nets in terms of matrices instead of using dual quaternions, as in Subsection~\ref{sec:snapping-nets}.

The rotation angle~$\alpha^t$ of~$\mathcal{R}_i^t$ is explicitly given by
\begin{equation}
  \label{eq:revolute-angle}
  \cos \frac{\alpha^t_\vvi}{2} = \frac{1 + t^2 \langle f_{\mu\vvi}^\star, f_{\pi\vvi}^\star\rangle}{\sqrt{(1+t^2\Vert f_{\mu\vvi}^\star\Vert^2)(1+t^2 \Vert f_{\pi\vvi}^\star \Vert^2)}}.
\end{equation}
It is the angle change between the two snapping configurations (the ``snap angle''). Moreover, the direction of the rotation axis in the fixed frame is
\begin{equation}
  \label{eq:axis-direction}
  u_\vvi = f_{\pi\vvi}^\star - f_{\mu\vvi}^\star + t f_{\pi\vvi}^\star \times f_{\mu\vvi}^\star.
\end{equation}
The axes themselves are incident with the vertices $f_\vvi^{-t}$.

\begin{remark}
  \begin{itemize}
  \item[(i)] It is remarkable that the snap angles \eqref{eq:revolute-angle} and the axis directions~\eqref{eq:axis-direction} depend only on the Koenigs dual~$f^*$ and are pointwise determined by only two values, $f^*_{\mu\vvi}$ and $f^*_{\pi\vvi}$. Thus, for our example of the discrete Enneper surface, these quantities are given in terms of the discrete Gauss map of the discrete minimal surface. As expected, the snap angles are small for small $t$ ($\lim_{t\to0}\alpha_\vvi = 0$) and they increase with $t$ locally around $t = 0$.
  \item[(ii)] Similar formulas can be obtained for face isometric surfaces not arising from diagonal nets of discrete Koenigs nets (see \cite{sauer37, sauer53, schief08} for the construction of face isometric surfaces). However, in this case the geometry of the discrete surfaces giving rise to the Cayley transforms is less lucid than for Koenigs duals.
  \end{itemize}
\end{remark}

\subsection{3D Printed Prototypes}
\label{sec:3d-print}

As a proof of concept, we fabricated 3D printed prototypes of a snapping four-bar net (see Figure~\ref{fig:prints}) and a ``snapping cube'', that is, the snapping structure corresponding to one cell of an $\SQ$-net of dimension $d = 3$. CAD data and videos of the snapping structures are available as supplementary material to this manuscript and via~\cite{suppl_zenodo}.

The data for the snapping four-bar net was generated from a discrete Enneper surface, as pointed out in Example~\ref{ex:enneper_surface}. To break symmetry and obtain more visually distinct shapes in the two stable configurations, we applied a Möbius transformation to the discrete Enneper surface. In the case of discrete minimal surfaces, this transformation again yields a discrete Koenigs net.

A snapping cube can be generated using methods similar to those presented in Section~\ref{sec:surface-design}. However, this construction requires controlling additional monodromy conditions during the infinitesimal isometric deformation. Further details will be provided in future work.

\begin{figure}
  \centering
  \includegraphics[scale=0.18]{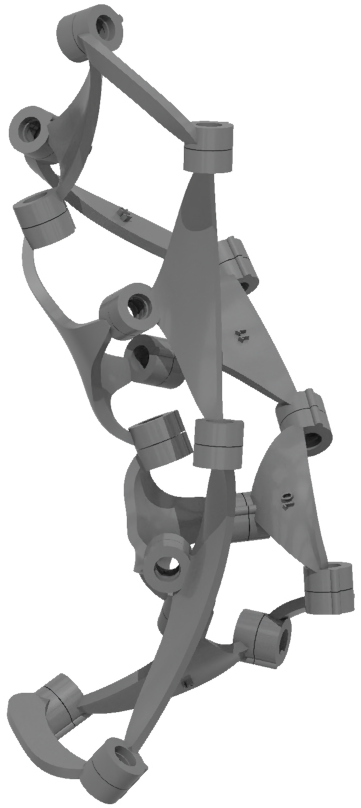}
  \includegraphics[scale=0.18]{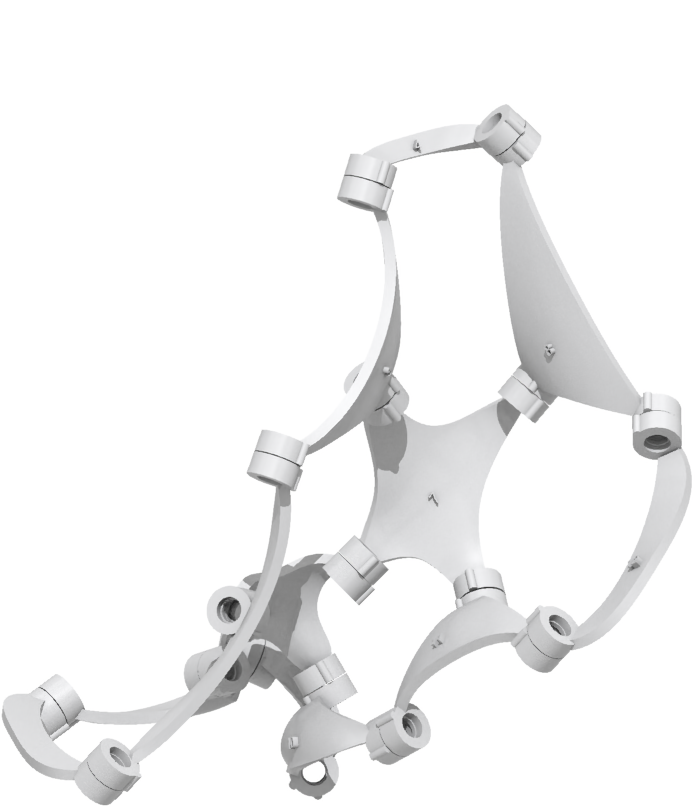}
  \qquad\includegraphics[scale=0.18]{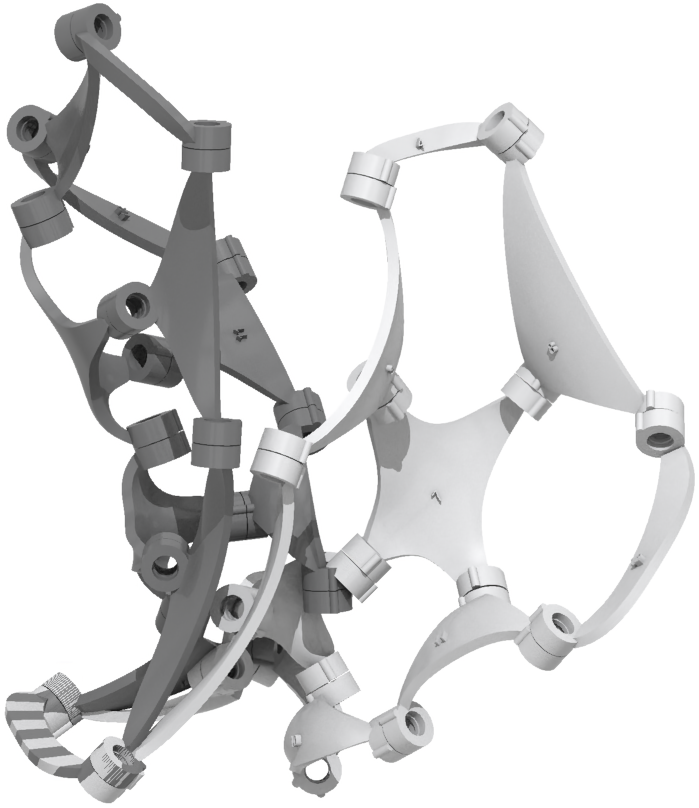}
  \caption{CAD renderings of a prototype snapping four-bar net in its two stable configurations. Numbers printed on the link parts are helpful during manual assembly.}
  \label{fig:prints}
\end{figure}

The CAD models were then created using the NURBS-modeling software Rhinoceros, which offers a wide range of options for shaping the rigid links between the joints. The parts were printed using a commercially available Bambu Lab 3D printer with standard PLA Material.

As input, each revolute joint was specified by
\begin{itemize}
\item a point, given by a vertex of one of the isometric discrete surface pair used for rolling, together with
\item the direction of the rotation axis in the corresponding stable configuration, c.f. \eqref{eq:axis-direction}.
\end{itemize}
In addition, the checkerboard combinatorics governing the connection of the revolute joints by rigid links were prescribed.

In the CAD design process, the first step was to model all joints as bearing housings~(Figure~\ref{fig:prints}). The bearings used in that model had an outer diameter of \qty{12}{\milli\metre} and a height of \qty{4}{\milli\metre}. To avoid geometric interference due to overlapping, some housings were shifted along their axes, which is always mechanically possible and does not affect the kinematic snapping behavior of the overall structure. This changes the original quadrangles of the underlying discrete surface but does not affect bistability.

In the second step, we designed the rigid links. In our initial attempt, we used pipe surfaces between adjacent joints, as shown in~\cite{huczala24}. However, the resulting snapping behavior was not as well defined as expected. We therefore adopted the more robust design presented here: spline curves connecting the joints served as edge curves for NURBS surfaces, which were then offset to assign material thickness. During assembly, screws matching the inner diameter of the bearings were inserted as axes, a washer was positioned between the two inner bearings, and the stack was fastened with a nut.

As shown in the accompanying video material, the prototypes exhibit the predicted snapping behavior.

\section{Conclusion}

We presented a mathematical treatment of snapping four-bar nets, a mechanical
structure that is bistable in an exact mathematical sense. By identifying snapping four-bar nets with discrete rolling motions and with quad nets in the Study quadric, we proved existence of arbitrarily large structures. Moreover, we suggested a construction method, based on infinitesimally flexible quad surfaces and Whiteley de-averaging, that allows to control axis positions as well as snap angles. We summarize the essential construction steps here in order to help readers who wish to construct their own models:
\begin{enumerate}
\item Pick a discrete quad net $f$ with a face or vertex rigid infinitesimal isometric deformation $q$. Many of those can be found in existing literature \cite{bobenko08, sauer33, wallner08}.
\item Apply Whiteley de-averaging to obtain a pair $f^{+t} = f + tq$, $f^{-t} = f - tq$ of isometric quad nets with congruent faces or with congruent vertex stars.
\item In case of congruent vertex stars, change the net combinatorics as illustrated in Figure~\ref{fig:graphs}.
\item Color the two surfaces in a checkerboard pattern and create a discrete rolling motion as illustrated in Figure~\ref{fig:rolling-surfaces}. The revolute axes in the fixed and in the moving frame, respectively, determine the two configurations of the snapping four-bar net.
\item In case of $f$ being a Koenigs net, use Equations~\eqref{eq:revolute-angle} to control the snap angle and \eqref{eq:axis-direction} to determine the axis directions directly.
\end{enumerate}

3D printed prototypes verify the feasibility of the thus obtained surface-like designs. Of course, actual mechanism synthesis, with clear constraints imposed on both configurations, is still a challenge.

Further open questions concern issues that go beyond the purely geometric aspects of snapping four-bar nets and include the following: Which deformations occur during the snap? Can they be modelled in such a way that collisions of links can be detected and avoided? Which force is needed to induce the snap and will the snap propagate through the net? What is the tension induced into the links during the snap? How should they be shaped and how can they be fabricated in order to ensure both, stiff target configurations and a long life-time?

Beyond these mechanical and physical aspects, additional questions of a purely geometric nature arise. For example, snapping four-bar nets obtained from face-isometric surfaces come in pairs, depending on whether the black or the white faces are used for rolling. How is the geometry of two such ``conjugated'' snapping mechanisms intertwined?

\section*{Acknowledgment}

The authors thank the anonymous reviewers for their helpful remarks and suggestions that led to a much improved revised version of this text.

\bibliographystyle{elsarticle-num}
\bibliography{references}
\end{document}